\documentclass[11pt, a4paper]{article}
\usepackage{amsmath, amssymb, enumerate, amsthm, natbib, accents}

\newtheorem{theorem}{Theorem}[section]
\newtheorem{proposition}{Proposition}[section]
\newtheorem{lemma}{Lemma}[section]
\newtheorem{corollary}{Corollary}[section]

\usepackage{typearea}
\typearea{12}

  \makeatletter
  \def\widebar{\accentset{{\cc@style\underline{\mskip10mu}}}}
  \makeatother

\begin{document}
	
	\title{Predictive Density Estimation under the Wasserstein Loss}
	\author{Takeru MATSUDA$^{\displaystyle 1}$ and William E. STRAWDERMAN$^{\displaystyle 2}$\\ \\
		$^{\displaystyle 1}$Department of Mathematical Informatics\\
		The University of Tokyo\\
		\texttt{\normalsize matsuda@mist.i.u-tokyo.ac.jp}\\
		$^{\displaystyle 2}$Department of Statistics and Biostatistics\\ Rutgers University}
		\date{}
	
	\maketitle

\begin{abstract}
We investigate predictive density estimation under the $L^2$ Wasserstein loss for location families and location-scale families.
We show that plug-in densities form a complete class and that the Bayesian predictive density is given by the plug-in density with the posterior mean of the location and scale parameters. 
We provide Bayesian predictive densities that dominate the best equivariant one in normal models.
\end{abstract}

\section{Introduction}
Suppose that we have $n$ independent observations $x^n=(x_1,\ldots,x_n)$ from a probability distribution $p (x \mid \theta)$ and predict the future observation $y$ from a probability distribution $\tilde{p} (y \mid \theta)$ by using a predictive density $\hat{p} (y \mid x^n)$, where $\theta$ is an unknown parameter.
Let $L(\tilde{p}(y \mid \theta),\hat{p}(y \mid x^n))$ be a loss function that measures the closeness of a predictive density $\hat{p}(y \mid x^n)$ to the true distribution $\tilde{p}(y \mid \theta)$.
Then, the risk function of the predicive density $\hat{p}(y \mid x^n)$ is defined as the expected loss:
\begin{align*}
	R(\theta,\hat{p}) = \int L(\tilde{p} (y \mid \theta), \hat{p}(y \mid x^n)) p (x^n \mid \theta) {\rm d} x^n.
\end{align*}
A predictive density $\hat{p}_*$ is called minimax if it minimizes the maximum risk:
\begin{align*}
\sup_{\theta} R(\theta,\hat{p}_*) = \min_{\hat{p}} \sup_{\theta} R(\theta,\hat{p}).
\end{align*}
A predictive density $\hat{p}_1$ is said to dominate another predictive density $\hat{p}_2$ if 
\begin{align*}
	R(\theta,\hat{p}_1) \leq R(\theta,\hat{p}_2)
\end{align*}
holds for every $\theta$ and also
\begin{align*}
	R(\theta,\hat{p}_1) < R(\theta,\hat{p}_2)
\end{align*}
holds for some $\theta$.
A predictive density $\hat{p}_{\pi}$ is called the Bayesian predictive density with respect to a prior $\pi (\theta)$ if it minimizes the average risk:
\begin{align*}
	\int R(\theta,\hat{p}_{\pi}) \pi (\theta) {\rm d} \theta = \min_{\hat{p}} \int R(\theta,\hat{p}) \pi (\theta) {\rm d} \theta.
\end{align*}

Predictive density estimation under the Kullback-Leibler loss has been well studied.
\cite{Aitchison} showed that the Bayesian predictive density with respect to a prior $\pi (\theta)$ is given by
\begin{align*}
	\hat{p}_{\pi} (y \mid x^n) = \int \tilde{p} (y \mid \theta) \pi (\theta \mid x^n) {\rm d} \theta
= \frac{\int \tilde{p} (y \mid \theta) p (x^n \mid \theta) \pi (\theta) {\rm d} \theta}{\int p (x^n \mid \theta) \pi (\theta) {\rm d} \theta}. \label{bayes_pred}
\end{align*}
For the normal model with known variance, the Bayesian predictive density based on the uniform prior is the best equivariant and also minimax \citep{Liang}.
\cite{Komaki01} showed that the Bayesian predictive density based on the Stein prior dominates that based on the uniform prior in dimension larger than or equal to three.
\cite{George06} generalized this result and proved that Bayesian predictive densities based on superharmonic priors dominate that based on the uniform prior.
\cite{Brown} gave a characterization of admissible predictive densities.
For the normal model with unknown variance, \cite{Kato} found a Bayesian predictive density that dominates the best equivariant one.
For general location-scale families, \cite{Komaki07} showed that the Bayesian predictive densities based on superharmonic priors asymptotically dominate the best equivariant one.

Predictive density estimation has been investigated for other loss functions as well.
\cite{Corcuera} studied the class of alpha divergence losses, which includes the Kullback-Leibler loss, and derived an explicit form of the Bayesian predictive density.
\cite{Kubokawa15} and \cite{Kubokawa17} considered $L^2$ loss and $L^1$ loss, respectively.
These studies showed that improved predictive density estimation depends on the loss function.

In this study, we investigate predictive density estimation when the loss function is defined by the Wasserstein distance, which is a distance function between probability measures based on the metric of the underlying space \citep{Olkin,Villani} and has been widely used in machine learning and computer vision \citep{Peyre}.
Since the Wasserstein distance is derived as the optimal transportation cost between distributions, predictive density estimation under the Wasserstein loss is suitable for resource allocation.
Namely, suppose that we predict the spatial distribution of demand for some resource and then distribute it accordingly.
To reduce the cost of re-transportation, predictive densities with smaller Wasserstein loss are preferable.
Here, we consider $L^2$ Wassestein distance and focus on location families and location-scale families.
For both families, plug-in densities are shown to form a complete class \citep{Berger}.
For location families, we prove that the Bayesian predictive density is given by the plug-in density with the posterior mean of the location parameter. 
For location-scale families, we prove that the Bayesian predictive density is given by the plug-in density with the posterior mean of the location and scale parameters. 
We give Bayesian predictive densities that dominate the best equivariant one in normal models.

This paper is organized as follows.
In Section 2, we present useful properties of the Wasserstein distance and also review improved estimation of location and scale parameters.
In Sections 3 and 4, we investigate predictive density estimation under the $L^2$ Wasserstien loss for location families and location-scale families, respectively.
In Section 5, we give concluding remarks.

\section{Preliminaries}
\subsection{Wasserstein distance}
The Wasserstein distance is a metric between probability distributions.
Let $S$ be a metric space with distance $d$.
The $L^p$ Wasserstein distance between two probability distributions $p_1$ and $p_2$ on $S$ is defined by
\[
	W_p (p_1,p_2) = \inf_{X,Y} \ {\rm E} [ d(X,Y)^p] ^{1/p},
\]
where the infimum is taken over all joint distributions of $(X,Y)$ with marginal distributions of $X$ and $Y$ equal to $p_1$ and $p_2$, respectively.
The Wasserstein distance can be interpreted as the optimal transportation cost from $p_1$ to $p_2$.
See \cite{Villani} for details.

In this study, we focus on the setting $S=\mathbb{R}^d$ and $p=2$.
Thus,
\[
	W_2 (p_1,p_2) = \inf_{X,Y} \ {\rm E} [ \| X-Y \|^2] ^{1/2}.
\]
For finiteness of the $L^2$ Wasserstein distance, we restrict our attention to distributions with finite mean and covariance in the following.

For normal distributions, the $L^2$ Wasserstein distance is obtained in closed form.

\begin{lemma}[\cite{Olkin}]
	The $L^2$ Wasserstein distance between $d$-dimensional normal distributions ${\rm N}_d(\mu_1,\Sigma_1)$ and ${\rm N}_d(\mu_2,\Sigma_2)$ is given by
	\[
	W_2({\rm N}_d(\mu_1,\Sigma_1),{\rm N}_d(\mu_2,\Sigma_2)) = \left( \| \mu_1 - \mu_2 \|^2 + {\rm tr} \left( \Sigma_1 + \Sigma_2 -  2 (\Sigma_1^{1/2} \Sigma_2 \Sigma_1^{1/2})^{1/2} \right) \right)^{1/2}.
	\]
\end{lemma}

Note that if $\Sigma_1$ and $\Sigma_2$ are commutative ($\Sigma_1 \Sigma_2=\Sigma_2 \Sigma_1$), then
\[
{\rm tr} \left( \Sigma_1 + \Sigma_2 -  2 (\Sigma_1^{1/2} \Sigma_2 \Sigma_1^{1/2})^{1/2} \right) = \| \Sigma_1^{1/2} - \Sigma_2^{1/2} \|_F^2,
\]
where $\| A \|_F = ({\rm tr} (A^{\top} A))^{1/2}$ denotes the Frobenius norm of $A$.

We will also use the following properties of the $L^2$ Wasserstein distance.

\begin{lemma}[\cite{Givens}]\label{lem_givens}
	Let $p_1$ and $p_2$ be probability densities on $\mathbb{R}^d$ with mean $\mu_1$ and $\mu_2$, respectively.
	Then,
	\[
	W_2(p_1,p_2) = \left( \| \mu_1 - \mu_2 \|^2 + W_2(q_1,q_2)^2 \right)^{1/2},
	\]
	where $q_1$ and $q_2$ are probability densities with mean zero defined by $q_1(x)=p_1(x+\mu_1)$ and $q_2(x)=p_2(x+\mu_2)$, respectively.
\end{lemma}

\begin{lemma}[\cite{Gelbrich}]\label{lem_gelbrich}
	Let $p_1$ and $p_2$ be probability densities on $\mathbb{R}^d$ with mean $\mu_1$ and $\mu_2$ and covariance matrices $\Sigma_1$ and $\Sigma_2$, respectively.
	Then,
\[
	W_2(p_1,p_2) \geq \left( \| \mu_1 - \mu_2 \|^2 + {\rm tr} \left( \Sigma_1 + \Sigma_2 -  2 (\Sigma_1^{1/2} \Sigma_2 \Sigma_1^{1/2})^{1/2} \right) \right)^{1/2}.
\]
The equality holds if
\begin{align*}
	p_2(x-\mu_2) = p_1  ( \tilde{\Sigma} (x-\mu_1)), \label{cond_eq}
\end{align*}
where $\tilde{\Sigma} =( \Sigma_1^{-1/2} (\Sigma_1^{1/2} \Sigma_2 \Sigma_1^{1/2})^{1/2} \Sigma_1^{-1/2})^{-1}$.
\end{lemma}


Unlike other dispersion measures between probability distributions such as the alpha divergence and the $L^p$ distance, the Wasserstein distance has the characteristic that it inherits the metric structure of the underlying space.
By exploiting this property, the Wasserstein distance has been widely used in machine learning and computer vision \citep{Peyre}.

\subsection{Shrinkage estimation of location parameter}
We briefly review existing results on Bayes shrinkage estimation of location parameters.
For more details, see \cite{shr_book}.

First, we consider the known scale cases.
Suppose that we have an observation $x \sim {\rm N}_d (\mu,I_d)$ and estimate $\mu$ under the quadratic loss $L(\mu,\hat{\mu}) = \| \hat{\mu}-\mu \|^2$.
For this problem, the usual estimator $\hat{\mu}(x)=x$ is minimax and the (generalized) Bayes estimator with respect to a prior $\pi(\mu)$ is given by the posterior mean:
\[
	\hat{\mu}^{\pi} (x) = \int \mu \pi(\mu \mid x) {\rm d} \mu = \frac{\int \mu p(x \mid \mu) \pi(\mu) {\rm d} \mu}{\int p(x \mid \mu) \pi(\mu) {\rm d} \mu}.
\]
Many minimax generalized Bayes estimators have been developed.
For example, \cite{Stein74} proved that the generalized Bayes estimator with respect to a superharmonic prior is minimax and dominates the usual estimator $\hat{\mu}(x)=x$. 
Analogous results for scale mixtures of normal distributions are given in \cite{Maruyama03} and \cite{Fourdrinier08b}.
Extensions to more general classes of spherically symmetric distributions are given in \cite{Fourdrinier08}.

Next, we consider the unknown scale cases.
Suppose that we have independent observations $x \sim {\rm N}_d \left( \mu, \sigma^2 I_d \right)$ and ${s} \sim \sigma^2 \chi^2_n$.
For the problem of estimating $\mu$ under the scaled quadratic loss $L(\mu,\hat{\mu};\sigma^2) = \| \hat{\mu}-\mu \|^2/\sigma^2$, the usual estimator $\hat{\mu}(x,s)=x$ is minimax and the (generalized) Bayes estimator with respect to a prior $\pi(\mu,\sigma)$ is given by
\[
\hat{\mu}^{\pi} (x,s) = \frac{\int \sigma^{-2} \mu \pi(\mu,\sigma \mid x,s) {\rm d} \mu {\rm d} \sigma}{\int \sigma^{-2} \pi(\mu,\sigma \mid x,s) {\rm d} \mu {\rm d} \sigma} = \frac{\int \sigma^{-2} \mu p(x,s \mid \mu,\sigma) \pi(\mu,\sigma) {\rm d} \mu {\rm d} \sigma}{\int \sigma^{-2} p(x,s \mid \mu,\sigma) \pi(\mu,\sigma) {\rm d} \mu {\rm d} \sigma}.
\]
\cite{Maruyama05} proposed a class of generalized Bayes minimax estimators in this setting.
By putting $a=-1$, $b=0$, $C=D=I_d$, $e=-1$ and $\gamma=1$ in their Theorem 2.3, the following is obtained.

\begin{lemma}[\cite{Maruyama05}]\label{lem_maru}
	Let $\tau=\sigma^{-2}$.
	If $(n-2)(d-4) \geq 8$, then the generalized Bayes estimator of $\mu$ under the scaled quadratic loss with respect to the hierarchical prior
	\[
	\theta \mid \lambda,\tau \sim {\rm N}_d (0,\tau^{-1} \lambda^{-1} (1-\lambda) I_d),
	\]
	\[
	\lambda \sim \lambda^{-1} 1_{[0,1]}, \quad \tau \sim \tau^{-1} 1_{(0,\infty)},
	\]
	 is minimax and dominates the usual estimator $\hat{\mu}(x,s)=x$.
\end{lemma}

Later, we will consider estimation of $\mu$ under the unscaled quadratic loss $L(\mu,\hat{\mu}) = \| \hat{\mu}-\mu \|^2$.
For this loss function, the usual estimator $\hat{\mu}(x,s)=x$ is the best equivariant and the (generalized) Bayes estimator with respect to a prior $\pi(\mu,\sigma)$ is given by
\[
\hat{\mu}^{\pi} (x,s) = \int \mu \pi(\mu,\sigma \mid x,s) {\rm d} \mu {\rm d} \sigma = \frac{\int \mu p(x,s \mid \mu,\sigma) \pi(\mu,\sigma) {\rm d} \mu {\rm d} \sigma}{\int p(x,s \mid \mu,\sigma) \pi(\mu,\sigma) {\rm d} \mu {\rm d} \sigma}.
\]
Then, Lemma \ref{lem_maru} is rewritten as follows.

\begin{lemma}\label{lem_maru2}
	Let $\tau=\sigma^{-2}$.
If $(n-2)(d-4) \geq 8$, then the generalized Bayes estimator of $\mu$ under the unscaled quadratic loss with respect to the hierarchical prior
\[
\theta \mid \lambda,\tau \sim {\rm N}_d (0,\tau^{-1} \lambda^{-1} (1-\lambda) I_d),
\]
\[
\lambda \sim \lambda^{-1} 1_{[0,1]}, \quad \tau \sim 1_{(0,\infty)},
\]
dominates the best equivariant estimator $\hat{\mu}(x,s)=x$.
\end{lemma}

\subsection{Improved estimation of scale parameter under quadratic loss}
Suppose that we have independent observations $x \sim {\rm N}_d \left( \mu, \sigma^2 I_d \right)$ and ${s} \sim \sigma^2 \chi^2_n$ where both $\mu$ and $\sigma^2$ are unknown.
\cite{Stein64} studied estimation of the variance $\sigma^2$ under the scaled quadratic loss $L(\sigma^2,\hat{\sigma}^2) = (\hat{\sigma}^2/\sigma^2-1)^2$ and showed that the estimator $\hat{\sigma}^2(x,s)=\min \{ s/(n+2), (\| x \|^2+s)/(n+d+2) \}$ dominates the best equivariant estimator $\hat{\sigma}_0^2(x,s)=s/(n+2)$.
Note that this improved estimator is not generalized Bayes.
Later, \cite{Brewster}, \cite{Strawderman} and \cite{Maruyama06} developed generalized Bayes minimax estimators that dominate $\hat{\sigma}_0^2(x,s)$.

Here, we consider estimation of the standard deviation $\sigma$ under the unscaled quadratic loss $L(\sigma,\hat{\sigma}) = (\hat{\sigma}-\sigma)^2$.
The best equivariant estimator is given by
\begin{align}
\hat{\sigma}_0(x,s)=c \sqrt{s}, \quad c= \frac{1}{\sqrt{2}} \frac{\Gamma \left( \frac{n+1}{2} \right)}{\Gamma \left( \frac{n+2}{2} \right)}. \label{c_def}
\end{align}
\cite{Kubokawa94} proposed a unified approach to improving on the best equivariant estimators of powers of the scale parameter.
Consider a class of estimators $\hat{\sigma}(x,s) = \phi(w) \sqrt{s}$ where $w=\| x \|^2/s$ and let
\[
\phi_0(w) = \frac{1}{\sqrt{2}} \frac{\Gamma \left( \frac{n+d+1}{2} \right)}{\Gamma \left( \frac{n+d+2}{2} \right)} \frac{\int_0^1 \lambda^{d/2-1} (1+\lambda w)^{-(n+d+1)/2} {\rm d} \lambda}{\int_0^1 \lambda^{d/2-1} (1+\lambda w)^{-(n+d+2)/2} {\rm d} \lambda}.
\]
Then, from Theorem 2.2 of \cite{Kubokawa94}, we obtain the following.

\begin{lemma}\label{lem_kubo}
	Assume the following conditions:
	\begin{itemize}
		\item $\phi(w)$ is non-decreasing.
		
		\item $\phi(w) \to c$ as $w \to \infty$, where $c$ is defined in \eqref{c_def}.
		
		\item $\phi(w) \geq \phi_0(w)$ for every $w \geq 0$.
	\end{itemize}
	Then, the estimator $\hat{\sigma}(x,s) = \phi(w) \sqrt{s}$ dominates $\hat{\sigma}_0(x,s)$ in \eqref{c_def}.
\end{lemma}

In fact, the lower bound $\phi_0(w)$ in Lemma \ref{lem_kubo} corresponds to a generalized Bayes estimator.

\begin{proposition}\label{prop_maru}
	Let $\tau=\sigma^{-2}$.
The estimator $\hat{\sigma}(x,s)=\phi_0(w) \sqrt{s}$ is the generalized Bayes estimator of $\sigma$ under the unscaled quadratic loss with respect to the following hierarchical prior:
\[
\theta \mid \lambda,\tau \sim {\rm N}_d (0,\tau^{-1} \lambda^{-1} (1-\lambda) I_d),
\]
\[
\lambda \sim \lambda^{-1} 1_{[0,1]}, \quad \tau \sim 1_{(0,\infty)}.
\]
\end{proposition}
\begin{proof}
The joint distribution of $(\tau,x,s)$ is
\begin{align*}
	& g (\tau,x,s) \\
	\propto & \int \int_0^1 \tau^{d/2} \exp \left( -\frac{\tau}{2} \| x-\theta \|^2 \right) \left( \frac{\tau \lambda}{1-\lambda} \right)^{d/2} \exp \left( -\frac{\tau \lambda}{2 (1-\lambda)} \| \theta \|^2  \right) \tau^{n/2} \exp \left( -\frac{\tau s}{2} \right) \lambda^{-1} {\rm d} \lambda {\rm d} \theta \\
	\propto & \tau^{(n+d)/2} \int_0^1 \left( \frac{\tau \lambda}{1-\lambda} \right)^{d/2} \exp \left( -\frac{\tau s}{2} \right) \lambda^{-1} \int \exp \left( -\frac{\tau}{2 (1-\lambda)} \| \theta-(1-\lambda)x \|^2 -\tau \frac{\| x\|^2 \lambda}{2} \right) {\rm d} \theta {\rm d} \lambda \\
	\propto & \tau^{(d+n)/2} \int_0^1 \lambda^{d/2-1} \exp \left( -\tau \frac{\| x \|^2 \lambda + s}{2} \right) {\rm d} \lambda.
\end{align*}
Therefore, the generalized Bayes estimator of $\sigma=\tau^{-1/2}$ under the unscaled quadratic loss is
\begin{align*}
	\hat{\sigma}^{\pi} (x,s) &= \frac{\int_0^{\infty} \tau^{-1/2} g(\tau,x,s) {\rm d} \tau}{\int_0^{\infty} g(\tau,x,s) {\rm d} \tau} \\
	&= \sqrt{\frac{s}{2}} \frac{\Gamma \left( \frac{n+d+1}{2} \right)}{\Gamma \left( \frac{n+d+2}{2} \right)} \frac{\int_0^1 \lambda^{d/2-1} (1+\lambda w)^{-(n+d+1)/2} {\rm d} \lambda}{\int_0^1 \lambda^{d/2-1} (1+\lambda w)^{-(n+d+2)/2} {\rm d} \lambda} \\
	&= \phi_0(w) \sqrt{s}.
\end{align*}
\end{proof}


\section{Location family}
Let
\[
	p(z \mid \mu) = f(z-\mu)
\]
be a location family with finite second moments, where $z \in \mathbb{R}^d$, $f(z) \geq 0$ and
\[
	\int f(z) {\rm d} z = 1, \quad \int z f(z) {\rm d} z = 0. 
\]
We consider prediction of $y \sim p(y \mid \mu)$ based on the observation $x \sim p (x \mid \mu)$ under the $L^2$ Wasserstein loss.

For this problem, the class of plug-in densities form a complete class \citep{Berger} as follows.
Note that it is sufficient to restrict our attention to predictive densities with finite mean, since the $L^2$ Wasserstein loss would diverge otherwise.

\begin{proposition}\label{prop_loc_complete}
Any predictive density $\hat{p}(y \mid x)$ with mean $\hat{\mu}(x)$ is dominated by the plug-in density $p(y \mid \hat{\mu} (x))$.
\end{proposition}
\begin{proof}
	By Lemma \ref{lem_givens}, we obtain
	\begin{align*}
	W_2(p(y \mid \mu),\hat{p}(y \mid x))^2 &= \| \mu-\hat{\mu}(x) \|^2 + W_2(p(y \mid \hat{\mu}(x)),\hat{p}(y \mid x))^2 \\
	& \geq \| \mu-\hat{\mu}(x) \|^2 \\
	&=	W_2(p(y \mid \mu),p(y \mid \hat{\mu}(x)))^2.
\end{align*}
Therefore, the risk of the plug-in density $p(y \mid \hat{\mu} (x))$ is not larger than that of $\hat{p}(y \mid x)$ for every $\mu$.
\end{proof}


Since $W_2(p(y \mid \mu),p(y \mid \hat{\mu}(x)))^2=\| \mu-\hat{\mu}(x) \|^2$, the risk function of the plug-in density $p(y \mid \hat{\mu} (x))$ is equal to the quadratic risk function of the estimator $\hat{\mu}(x)$:
\begin{align*}
R(\mu,\hat{p}) = {\rm E}_{\mu} \| \hat{\mu}(x)-\mu \|^2. \label{risk_eq_loc}
\end{align*}
Thus, predictive density estimation under the $L^2$ Wasserstein loss reduces to point estimation of $\mu$ under the quadratic loss.
Therefore, the best equivariant predictive density and the Bayesian predictive density is obtained as follows.

\begin{theorem}\label{prop_loc}
	The best equivariant predictive density under the $L^2$ Wasserstein loss is given by the plug-in density with the usual estimator $\hat{\mu}(x)=x$ and it is minimax.
\end{theorem}

\begin{theorem}\label{th_loc}
The Bayesian predictive density with respect to a prior $\pi(\mu)$ under the $L^2$ Wasserstein loss is given by the plug-in density with the posterior mean $\hat{\mu}^{\pi}(x)$ of $\mu$:
	\[
	\hat{p}_{\pi} (y \mid x) = f(y-\hat{\mu}^{\pi} (x)),
\]
	where
\[
\hat{\mu}^{\pi} (x) = \int \mu \pi(\mu \mid x) {\rm d} \mu = \frac{\int \mu p(x \mid \mu) \pi(\mu) {\rm d} \mu}{\int p(x \mid \mu) \pi(\mu) {\rm d} \mu}.
\]
Its risk function is equal to the quadratic risk function of the Bayes estimator $\hat{\mu}^{\pi}(x)$:
\begin{align*}
	R(\mu,\hat{p}_{\pi}) = {\rm E}_{\mu} \| \hat{\mu}^{\pi}(x)-\mu \|^2. 
\end{align*}
\end{theorem}
\begin{proof}
From Proposition \ref{prop_loc_complete}, it is sufficient to restrict our attention to the class of plug-in densities.
For this class, the risk function is obtained as
\begin{align*}
	R(\mu,\hat{p}) = \int W_2(p(y \mid \mu),\hat{p}(y \mid x))^2 p(x \mid \mu) {\rm d} x &= \int \| \mu-\hat{\mu}(x) \|^2 p(x \mid \mu) {\rm d} x.
\end{align*}
which is equal to the quadratic risk of $\hat{\mu}$ as an estimator of $\mu$.
Therefore, the average risk is minimized by setting $\hat{\mu}$ to the Bayes estimator $\hat{\mu}^{\pi}$ under the quadratic loss, which is given by the posterior mean \citep{Lehmann}.
\end{proof}


\begin{corollary}\label{prop_loc2}
	If the (generalized) Bayes estimator of $\mu$ with respect to a prior $\pi(\mu)$ is minimax under the quadratic loss, then the Bayesian predictive density with respect to the prior $\pi(\mu)$ is also minimax under the $L^2$ Wasserstein loss.
\end{corollary}

Therefore, Bayesian predictive densities with respect to shrinkage priors are minimax.
For example, from \cite{Stein74}, we obtain the following.

\begin{corollary}\label{cor_loc}
	For the normal model with known variance, the Bayesian predictive density with respect to a superharmonic prior $\pi(\mu)$ is minimax under the $L^2$ Wasserstein loss.
\end{corollary}

Similarly, improved predictive densities are obtained for scale mixtures of normal distributions by using results in \cite{Maruyama03} and \cite{Fourdrinier08b} and also for spherically symmetric distributions by using results in \cite{Fourdrinier08}.

\section{Location-scale family}
Let
\[
	p(z \mid \mu,\sigma) = \frac{1}{\sigma} f \left( \frac{z-\mu}{\sigma} \right)
\]
be a location-scale family, where $z \in \mathbb{R}^d$, $f(z) \geq 0$ and
\[
	\int f(z) {\rm d} z = 1, \quad \int z f(z) {\rm d} x = 0, \quad \int z z^{\top} f(z) {\rm d} z = I_d.
\]
We consider prediction of $y \sim p(y \mid \mu,\sigma)$ based on the independent observations $x_1,\ldots,x_n \sim p (x \mid \mu,\sigma)$ under the $L^2$ Wasserstein loss.

For this problem, the class of plug-in densities form a complete class \citep{Berger} as follows.
Note that it is sufficient to restrict our attention to predictive densities with finite mean and covariance, since the $L^2$ Wasserstein loss would diverge otherwise.

\begin{proposition}\label{prop_locs_complete}
	Any predictive density $\hat{p}(y \mid x^n)$ with mean $\hat{\mu}(x^n)$ and covariance $\hat{\Sigma}(x^n)$ is dominated by the plug-in density $p(y \mid \hat{\mu} (x^n),\hat{\sigma}(x^n))$ where $\hat{\sigma}(x^n) = d^{-1} {\rm tr} (\hat{\Sigma}(x^n)^{1/2})$.
\end{proposition}
\begin{proof}
By Lemma \ref{lem_gelbrich}, we obtain
\begin{align*}
W_2(p(y \mid \mu,\sigma),\hat{p}(y \mid x^n))^2 &= \| \hat{\mu}(x^n)-\mu \|^2 + \| \hat{\Sigma}(x^n)^{1/2}-\sigma I \|_{{\rm F}}^2 \nonumber \\
& \quad + W_2((\det \hat{\Sigma}(x^n))^{-1/2} f(\hat{\Sigma}(x^n)^{-1/2} (y-\hat{\mu}(x^n))),\hat{p}(y \mid x^n))^2 \\
& \geq  \| \hat{\mu}(x^n)-\mu \|^2 + \| \hat{\Sigma}(x^n)^{1/2}-\sigma I \|_{{\rm F}}^2 \\
& \geq  \| \hat{\mu}(x^n)-\mu \|^2 + \| \hat{\sigma}(x^n)I_d-\sigma I_d \|_{{\rm F}}^2 \\
& = W_2(p(y \mid \mu,\sigma),{p}(y \mid  \hat{\mu} (x^n),\hat{\sigma}(x^n)))^2,
\end{align*}
where we used 
\begin{align*}
\| \hat{\Sigma}(x^n)^{1/2}-\sigma I_d \|_{{\rm F}}^2 &\geq \sum_{k=1}^d ((\hat{\Sigma}(x^n)^{1/2})_{kk} -\sigma)^2 \geq d \left( \frac{1}{d} {\rm tr}(\hat{\Sigma}(x^n)^{1/2}) -\sigma \right)^2  = \| \hat{\sigma}(x^n) I-\sigma I_d \|_{{\rm F}}^2
\end{align*}
in the second inequality.
Therefore, the risk of the plug-in density $p(y \mid \hat{\mu} (x^n),\hat{\sigma} (x^n))$ is not larger than that of $\hat{p}(y \mid x^n)$ for every $(\mu,\sigma)$.
\end{proof}

Since $W_2(p(y \mid \mu,\sigma),{p}(y \mid  \hat{\mu} (x^n),\hat{\sigma}(x^n)))^2 = \| \hat{\mu}(x^n)-\mu \|^2 + d (\hat{\sigma}(x^n)-\sigma)^2$, the risk function of the plug-in density $p(y \mid \hat{\mu} (x^n),\hat{\sigma}(x^n))$ is equal to the sum of the (unscaled) quadratic risk functions of the estimators $\hat{\mu}(x^n)$ and $\hat{\sigma}(x^n)$:
\begin{align*}
R((\mu,\sigma),\hat{p}) = {\rm E}_{\mu,\sigma} \| \hat{\mu}(x^n)-\mu \|^2 + d \cdot {\rm E}_{\mu,\sigma} (\hat{\sigma}(x^n)-\sigma )^2.
\end{align*}
Thus, predictive density estimation under the $L^2$ Wasserstein loss reduces to joint estimation of $\mu$ and $\sigma$ under unscaled quadratic losses.
Therefore, the best equivariant predictive density and the Bayesian predictive density is obtained as follows.

\begin{theorem}\label{prop_locs}
	The best equivariant predictive density under the $L^2$ Wasserstein loss is given by the plug-in density with the best equivariant estimators of $\mu$ and $\sigma$ under the unscaled quadratic losses.
\end{theorem}

\begin{theorem}\label{th_locs}
	The Bayesian predictive density with respect to a prior $\pi(\mu,\sigma)$ under the $L^2$ Wasserstein loss is given by the plug-in density with the posterior mean $(\hat{\mu}^{\pi}(x^n),\hat{\sigma}^{\pi}(x^n))$ of $(\mu,\sigma)$:
	\[
\hat{p}_{\pi} (y \mid x^n) = \frac{1}{\hat{\sigma}^{\pi}(x^n)} f \left( \frac{y-\hat{\mu}^{\pi} (x^n)}{\hat{\sigma}^{\pi}(x^n)} \right),
\]
where
\[
\hat{\mu}^{\pi} (x^n) = \int \mu \pi(\mu,\sigma \mid x^n) {\rm d} \mu {\rm d} \sigma = \frac{\int \mu p(x^n \mid \mu,\sigma) \pi(\mu,\sigma) {\rm d} \mu {\rm d} \sigma}{\int p(x^n \mid \mu,\sigma) \pi(\mu,\sigma) {\rm d} \mu {\rm d} \sigma},
\]
\[
\hat{\sigma}^{\pi} (x^n) = \int \sigma \pi(\mu,\sigma \mid x^n) {\rm d} \mu {\rm d} \sigma = \frac{\int \sigma p(x^n \mid \mu,\sigma) \pi(\mu,\sigma) {\rm d} \mu {\rm d} \sigma}{\int p(x^n \mid \mu,\sigma) \pi(\mu,\sigma) {\rm d} \mu {\rm d} \sigma}.
\]
Its risk function is equal to the sum of the (unscaled) quadratic risk functions of $\hat{\mu}^{\pi}(x^n)$ and $\hat{\sigma}^{\pi}(x^n)$:
	\begin{align*}
R((\mu,\sigma),\hat{p}_{\pi}) = {\rm E}_{\mu,\sigma} \| \hat{\mu}^{\pi}(x^n)-\mu \|^2 + d \cdot {\rm E}_{\mu,\sigma} (\hat{\sigma}^{\pi}(x^n)-\sigma)^2. \label{risk_eq_locs}
\end{align*}
\end{theorem}
\begin{proof}
	From Proposition \ref{prop_locs_complete}, it is sufficient to restrict our attention to the class of plug-in densities $\hat{p}(y \mid x^n) = p(y \mid \hat{\mu}(x^n), \hat{\sigma}(x^n))$.
	For this class, the risk function is obtained as
	\begin{align*}
	R((\mu,\sigma),\hat{p}) &= \int W_2(p(y \mid \mu,\sigma),{p}(y \mid \hat{\mu}(x^n), \hat{\sigma}(x^n)))^2 p(x^n \mid \mu,\sigma) {\rm d} x^n \\
	&= \int \| \hat{\mu}(x^n)-\mu \|^2 p(x^n \mid \mu,\sigma) {\rm d} x^n+d \int ( \hat{\sigma}(x^n)-\sigma )^2 p(x^n \mid \mu,\sigma) {\rm d} x^n,
	\end{align*}
	which is equal to the sum of the quadratic risks of $\hat{\mu}$ and $\hat{\sigma}$ as estimators of $\mu$ and $\sigma$, respectively.
	Therefore, the average risk is minimized by setting $\hat{\mu}$ and $\hat{\sigma}$ to the Bayes estimators $\hat{\mu}^{\pi}$ and $\hat{\sigma}^{\pi}$ under the unscaled quadratic loss, which are given by the posterior mean \citep{Lehmann}.
\end{proof}

\begin{corollary}\label{prop_locs2}
	If both of the Bayes estimators of $\mu$ and $\sigma$ with respect to a prior $\pi(\mu,\sigma)$ dominate the best equivariant estimator under the unscaled quadratic losses, then the Bayesian predictive density with respect to the prior $\pi(\mu,\sigma)$ dominates the best equivariant predictive density under the $L^2$ Wasserstein loss.
\end{corollary}

Finally, we provide a concrete example of the Bayesian predictive density that dominates the best equivariant one in the normal model with unknown variance.
Specifically, let $x_1,\ldots,x_n \sim {\rm N}_d (\mu,\sigma^2 I_d)$ independently and consider prediction of $y \sim {\rm N}_d (\mu,\sigma^2 I_d)$.
The sufficient statistics are 
\[
\bar{x} = \frac{1}{n} \sum_{i=1}^n x_i \sim {\rm N}_d \left( \mu, \frac{\sigma^2}{n} \right),
\]
and
\[
s = \sum_{i=1}^n (x_i - \bar{x})^2 \sim \sigma^2 \chi_{n-1}^2.
\]
From Theorem \ref{prop_locs}, the best equivariant predictive density is $\hat{p} (y \mid x^n)={\rm N}_d(\bar{x},c^2 {s} I_d)$ where $c$ is defined in \eqref{c_def}.
By combining Lemma \ref{lem_maru2} and Proposition \ref{prop_maru} with Corollary \ref{prop_locs2}, an improved Bayesian predictive density is obtained as follows.

\begin{theorem}\label{th_normal}
	If $(n-3)(p-4) \geq 8$, then the Bayes solution with respect to the following hierarchical prior dominates the best equivariant predictive density under the $L^2$ Wasserstein loss:
\[
\theta \mid \lambda,\tau \sim {\rm N}_d (0,\tau^{-1} \lambda^{-1} (1-\lambda) I_d),
\]
\[
\lambda \sim \lambda^{-1} 1_{[0,1]}, \quad \tau \sim 1_{(0,\infty)}.
\]
\end{theorem}

\section{Conclusion}
In this study, we investigated predictive density estimation under the $L^2$ Wasserstein loss.
For both families, plug-in densities form a complete class \citep{Berger}.
For location families, the Bayesian predictive density is given by the plug-in density with the posterior mean of the location parameter. 
For location-scale families, the Bayesian predictive density is given by the plug-in density with the posterior mean of the location and scale parameters. 
We provided Bayesian predictive densities that dominate the best equivariant one in normal models.

We focused on location families and location-scale families in this study.
Extension to the case of an unknown covariance matrix is an interesting future problem.
Namely, we consider the model
\[
	p(x \mid \mu,\Sigma) = \frac{1}{(\det \Sigma)^{1/2}} f \left( \Sigma^{-1/2} (x-\mu) \right),
\]
which includes the elliptically contoured distributions \citep{Fang} as special cases.
In this case, the problem reduces to point estimation of the covairance matrix $\Sigma$ under the loss $L(\Sigma,\hat{\Sigma}) = W_2 (p(x \mid 0,\Sigma),p(x \mid 0,\hat{\Sigma}))^2 = {\rm tr} \left( \Sigma + \hat{\Sigma}-  2 (\Sigma^{1/2} \hat{\Sigma} \Sigma^{1/2})^{1/2} \right)$, which is different from Stein's loss or Frobenius loss.
In addition, while we considered $L^2$ Wasserstein loss in this study, 
it is an interesting future work to study predictive density estimation under other Wasserstein losses.

\section*{Acknowledgements}
William Strawderman's research is partially supported by a grant from the
Simons Foundation (\#418098).

\end{document}